\documentclass[10pt, reqno]{amsart}
\usepackage{amsmath,amssymb,amsfonts, amscd,hyperref}
\usepackage{color}
\usepackage[english]{babel}

\newcommand{\Hmm}[1]{\leavevmode{\marginpar{\tiny%
$\hbox to 0mm{\hspace*{-0.5mm}$\leftarrow$\hss}%
\vcenter{\vrule depth 0.1mm height 0.1mm width \the\marginparwidth}%
\hbox to 0mm{\hss$\rightarrow$\hspace*{-0.5mm}}$\\\relax\raggedright
#1}}}

\newtheorem{thm}{Theorem}[section]
\newtheorem{cor}{Corollary}

\newtheorem{lemma}[thm]{Lemma}
\newtheorem{pro}[thm]{Proposition}

\theoremstyle{definition}

\newtheorem{eg}[thm]{Example}
\newtheorem{rem}[thm]{Remark}
\newtheorem*{Remark}{Remark}

\newcommand{\R}{{\mathbb R}}
\newcommand{\C}{{\mathbb C}}

\newcommand{\N}{{\mathbb N}}

\newcommand{\al}{{\alpha}}
\newcommand{\be}{{\beta}}
\newcommand{\de}{{\delta}}
\newcommand{\ph}{{\varphi}}
\newcommand{\eps}{{\varepsilon}}

\newcommand{\si}{{\sigma}}
\newcommand{\ka}{{\kappa}}
\newcommand{\lm}{{\lambda}}

\newcommand{\Lip}{{\mathrm{Lip}_{\eps}^{\infty}}}

\newcommand{\ov}[1]{\overline{ #1}}
\newcommand{\ow}[1]{\widetilde{ #1}}

\begin{document}

\title{On the $l^p$ spectrum of Laplacians on graphs}

\author[F. Bauer]{Frank Bauer}
\address{Frank Bauer, Department of Mathematics, Harvard University, One Oxford Street  Cambridge, MA 02138, USA and  Max-Planck-Institut f\"ur Mathematik in den Naturwissenschaften,
Inselstr. 22, 04103 Leipzig, Germany }
\email{fbauer@math.harvard.edu}

\author[B. Hua]{Bobo Hua}
\address{Bobo Hua, Max-Planck-Institut f\"ur Mathematik in den Naturwissenschaften,
Inselstr. 22, 04103 Leipzig, Germany} \email{bobo.hua@mis.mpg.de}

\author[M. Keller]{Matthias Keller}
\address{Matthias Keller, Einstein Institute of Mathematics, The Hebrew University of Jerusalem,
Jerusalem 91904, Israel} \email{mkeller@ma.huji.ac.il}

\date{\today}



\begin{abstract}We study the $p$-independence of
spectra of Laplace operators on graphs arising from regular
Dirichlet forms on discrete spaces. Here, a sufficient criterion is
given solely by a uniform subexponential growth condition. Moreover,
under a mild assumption on the measure we show a {one-sided
spectral} inclusion without any further assumptions. We study
applications to normalized Laplacians including symmetries of the
spectrum and a characterization for positivity of the Cheeger
constant. Furthermore, we consider Laplacians on planar
tessellations for which we relate the spectral $p$-independence to
assumptions on the curvature.
\end{abstract}

 \maketitle

 \section{Introduction}

In \cite{Simon80,Simon82} Simon conjectured that the spectrum of a
Schr\"{o}dinger operator acting on $L^p(\R^N)$ is $p$-independent.
Hempel and Voigt gave an affirmative answer in \cite{HempelVoigt86}
for a large class of potentials. Later this result was generalized
in various ways. Sturm \cite{Sturm93} showed $p$-independence of the
spectra for uniformly elliptic operators on a complete Riemannian
manifold with uniform subexponential volume growth and a lower bound
on the Ricci curvature. Moreover, Arendt \cite{Arendt94} proved
$p$-independence of the  spectra of uniformly elliptic operators in
$\R^N$ with Dirichlet or Neumann boundary conditions under the
assumption of upper Gaussian estimates for the corresponding
semigroups. While the proof strategies of Sturm and Arendt are
rather similar to the one used by Hempel and Voigt, Davies
\cite{Davies95} gave a simpler proof of the $p$-independence of the
spectrum under the stronger assumption of polynomial volume growth
and Gaussian upper bounds using the functional calculus developed in
\cite{Davies95b}. In recent works   $p$-independence of spectral
bounds are proven in the context of conservative Markov processes
\cite{Kusuoka,Takeda} and Feynman-Kac semigroups
\cite{Chen,DKK,TakedaTawara}. See also \cite{GHKLW} for  Laplace
operators on graphs with finite measure.

In this paper, we prove  $p$-independence of spectra for Laplace
operators on graphs under the assumption of uniform subexponential
volume growth. This  question was brought up in
\cite[page~378]{Davies07} by Davies. Our  framework are regular
Dirichlet forms on discrete sets as introduced in \cite{KL}. While
our result is similar to the one of Sturm {\cite{Sturm93}} for
elliptic operators on manifolds, we do not need to assume any type
of lower curvature bounds nor any type of bounded geometry. However,
in various classical examples, such as  Laplacians with standard
weights, this disparity is resolved by the fact that  uniform
subexponential growth implies bounded geometry in some cases. In
further contrast to \cite{Sturm93}, we do not assume any uniformity
of the coefficients in the divergence part of the operator such as
uniform ellipticity and, additionally, we allow for positive
potentials {(in general potentials bounded from below)}.

We overcome the difficulties resulting from unbounded geometry, by
the use of intrinsic metrics. While this concept is well established
for strongly local Dirichlet forms,  \cite{Sturm94}  it  was only
recently introduced for general regular Dirichlet forms by
Frank/Lenz/Wingert in \cite{FLW}. Since then, this concept already
proved to be very effective for the analysis on graphs, see
\cite{BKW,FOLZ,FOLZ2,GHM,HKW,Hu,Hu2, HKMW} where it also appears
under the name adapted metrics. Moreover, we employ rather weak heat
kernel estimates (with the log term instead of a square) by Folz,
\cite{FOLZ}, which is a generalization of \cite{Davies93} by Davies.
These weak estimates turn out to be sufficient to prove the
$p$-independence. Of course, the condition on uniform subexponential
growth is always expressed with respect to an intrinsic metric.

Another result of this paper is the inclusion of the
$\ell^{2}$-spectrum in the $\ell^{p}$-spectrum under the assumption
of lower bounds on the measure only.

As applications we discuss the normalized Laplace operator, for
which we prove several basic properties of the $\ell^{p}$ spectra
such as certain symmetries of the spectrum. Moreover, we discuss
consequences of ${p}$-independence on the Cheeger constant and give
an example of ${p}$-independence and superexponential volume growth.
Finally, we consider the case of planar tessellations which relates
curvature bounds to the volume growth. In particular, we use such
curvature conditions to recover results of Sturm, \cite{Sturm93} in
the setting of planar tessellations.

The paper is organized as follows. In the next section we introduce
the set up and present the main results. In
Section~\ref{s:preliminaries} we show several auxiliary results in
order to prove the main results in Sections~\ref{s:proof_usg}
and~\ref{s:proof_upm}. Applications to  normalized Laplacians are
considered in Section~\ref{s:normalized}. The final section,
Section~\ref{s:tessellations}, is devoted to planar tessellations
and consequences of  curvature bounds on the volume growth and
$p$-independence.


\section{Set up and main results}


\subsection{Graphs}
Assume that $X$ is a countable set equipped with the discrete
topology. A strictly positive function $m:X\to(0,\infty)$ gives a
Radon measure on $X$ of full support via $m(A)=\sum_{x\in A}m(x)$
for $A \subseteq X$, so that $(X,m)$ becomes a discrete measure
space.

A graph over $(X,m)$ is a pair $(b,c)$. Here, $c:X\to[0,\infty)$ and
$b:X\times X\to[0,\infty)$ is a  symmetric function  with zero
diagonal that satisfies
$$
\sum_{y\in  X}b(x, y)<\infty \quad \textup{ for  } x\in X.$$ We say
$x$ and $y $ are \emph{neighbors} or \emph{connected} by an edge if
$b(x,y)>0$ and we write $x\sim y$.  For convenience we assume that
there are no isolated vertices, i.e., every vertex has a neighbor.
We call $b$ \emph{locally finite} if each vertex has only finitely
many neighbors. The function $c$ can be interpreted either as
one-way-edges to infinity or a potential or a killing term.

The \emph{normalizing measure} $n:X\to(0,\infty)$ given by
\begin{align*}
    n(x)=\sum_{y\in X}b(x,y), \quad \textup{ for  } x\in X.
\end{align*}
often plays a distinguished role.  In the case where $b:X\times
X\to\{0,1\}$, $n(x)$ gives the number of neighbors of a vertex $x$.
If $n/m\le M$ for some fixed $M>0$ we say the graph has
\emph{bounded geometry}.


\subsection{Intrinsic metrics and uniform subexponential growth}
By a \emph{pseudo metric} we understand a function $d:X\times
X\to[0,\infty)$ that is symmetric, has zero diagonal and satisfies
the triangle inequality. Following \cite{FLW}, we call a pseudo
metric $d$  an \emph{intrinsic metric} for a graph $b$ on $(X,m)$ if
\begin{align*}
 \sum_{y\in X}b(x,y)d(x,y)^{2}\leq m(x)\quad \mbox{ for all } x\in X.
\end{align*}
For example one can always choose the path metric induced by the
edge weights $w(x,y)=((m/n)(x)\wedge (m/n)(y))^{\frac{1}{2}}$, for
$x\sim y,$ cf. e.g. \cite{Hu}. Moreover, we call $$s:=\sup\{d(x,y)\
|\ x\sim y, x,y\in X\}$$ the \emph{jump size} of $d$. Note that the
natural graph metric $d_{n}$ (i.e., the path metric with weights
$w(x,y)=1$ for $x\sim y$) is  intrinsic  if and only if $m\ge n$.
However, in the case of {bounded geometry}, i.e., $n/m\le M$ for
some fixed $M>0$, the metric $d_{n}/\sqrt{M}$ (which is equivalent
to $d_{n}$) is an intrinsic metric.
\medskip

Throughout the paper we assume that $d$ is an intrinsic metric with
finite jump size. For the remainder of the paper, we refer to  the
quintuple $(X,b,c,m,d)$ whenever we speak of \emph{the graph}.
\medskip

We denote the distance balls centered at a vertex $x\in X$ with
radius $r\ge0$ by $B_{r}(x):=\{y\in X\mid d(x,y)\leq r\}$. Similar
to \cite{Sturm93}, we say the graph has uniform \emph{subexponential
growth} if for all $\eps>0$ there is {$C_{\eps}>0$ such that
\begin{align*}
    m(B_{r}(x))\leq C_{\eps}e^{\eps r} m(x),\qquad\mbox{for all $x\in X$, $r\ge0$}.
\end{align*}
In Section~\ref{s:growth} we discuss some implications of this
assumption.


\subsection{Dirichlet forms and Graph Laplacians}\label{s:operators}

Denote by  $C_{c}(X)$ the space  of complex valued functions on  $X$
with compact support. Denote the $\ell^{p}$-spaces by
\begin{align*}
\ell^{p}:=\ell^{p}(X,m)&:=\big\{f:X\to\C\mid
\|f\|_{p}<\infty\big\},\quad p\in[1,\infty],
\end{align*}
where
\begin{align*}
\|f\|_{\infty}:=\sup_{x\in X}|f(x)|,\quad\mbox{and}\quad
\|f\|_{p}:=\Big(\sum_{x\in X}|f(x)|^{p}m(x)\Big)^{\frac{1}{p}},\;
p\in[1,\infty).
\end{align*}
Note that $\ell^{\infty}(X,m)$ does not depend on $m$.

For $p\in[1,\infty]$, let the H\"older conjugate be denoted by
$p^{*}$, that is $\frac{1}{p}+\frac{1}{p^{*}}=1$. We denote the dual
pairing of $f\in \ell^{p}(X,m)$, $g\in \ell^{p^{*}}(X,m)$ by
\begin{align*}
    \langle f,g\rangle:=\sum_{x\in X}f(x)\overline{g(x)}m(x),
\end{align*}
which becomes a scalar product for $p=2$. We define the
sesqui-linear form $Q$ with domain $D(Q)\subseteq \ell^{2}$ by
\begin{align*}
    Q(f,g)=\frac{1}{2}\sum_{x,y\in X}b(x,y)(f(x)&-f(y))\overline{(g(x)-g(y))}+\sum_{x\in X} c(x)f(x)\overline{g(x)},\\
    D(Q)&=\ov{C_{c}(X)}^{\|\cdot\|_{Q}},
\end{align*}
where $\|\cdot\|_{Q}=(Q(\cdot)+\|\cdot\|^{2}_{2})^{\frac{1}{2}}$ and
$Q(f)=Q(f,f)$. The form $Q$ is a regular Dirichlet form on
$\ell^{2}(X,m)$, see \cite{FOT,KL} and for complexification of the
forms, see \cite[Appendix~B]{HKLW}. The corresponding positive
{selfadjoint operator $L=L_{2}$} on $\ell^{2}(X,m)$ acts as
\begin{align*} Lf(x)=\frac{1}{m(x)}\sum_{y\in
X}b(x,y)(f(x)-f(y))+\frac{c(x)}{m(x)}f(x).
\end{align*}  Let $\ow L$ be the
extension of $L$ to {$$\ow F=\{f:X\to\C\mid\sum_{y\in
X}b(x,y)|f(y)|<\infty\mbox{ for all }x\in X\}.$$} We have
$C_{c}(X)\subseteq D(L)$ if (and only if) $\ow L C_{c}(X)\subseteq
\ell^{2}(X,m)$, see \cite[Theorem~6]{KL}. In particular, this can
easily seen to be the case if the graph is locally finite or if
$\inf_{x\in X}m(x)>0$. {If $m=n$} and $c\equiv0$, then $L$ is
referred to as the \emph{normalized Laplacian}.

Moreover, $L=L_{2}$ gives rise to the resolvents
$G_{\al}=(L-\al)^{-1}$, $\al<0$ and the semigroups $T_{t}=e^{-tL}$,
$t\ge0$. These operators are positivity preserving and contractive
(as $Q$ is a Dirichlet form), and therefore extend consistently to
operators on $\ell^{p}(X,m)$, $p\in [1,\infty]$ (either by monotone
convergence or by density of $\ell^{2}\cap\ell^{p}$ in $\ell^{p}$,
$p\in [1,\infty)$ and taking the dual operator on $\ell^{1}$ to get
the operator on $\ell^{\infty}$). The semigroups are strongly
continuous for $p<\infty$. See \cite[Theorem~1.4.1]{Davies89} for a
proof of these facts.

We denote the positive generators of $G_{\al}$ or $T_{t}$ on
$\ell^{p}$ by $L_{p}$, $p\in[1,\infty)$, and the dual operator
${L_{1}}^{*}$ of $L_{1}$ on $\ell^{\infty}$ by $L_{\infty}$ (which
normally does not has dense domain in $\ell^{\infty}$). By
\cite[Theorem~9]{KL}) we have that $L_{p}$, $p\in[1,\infty]$, are
restrictions of $\ow L$. Moreover, $L_{p}$ are bounded operators
with norm bound $2C$ if $(n+c)/m\leq C$, see e.g.
\cite[Theorem~11]{KL2} or \cite[Theorem~9.3]{HKLW}. Hence, bounded
geometry is equivalent to boundedness of the operators for
$c\equiv0$. We denote the spectrum of the operator $L_p$ by
$\si(L_{p})$ and the resolvent set by
$\rho(L_{p})=\C\setminus \si(L_{p})$, $p\in[1,\infty]$. By duality $\si(L_{p})=\si(L_{p^{*}})$, $p\in[1,\infty]$.\\
Throughout this paper $C$ always denotes a constant that might
change from line to line.
\subsection{Main Results}

In this section, we state the main theorems of this paper. The first
is the discrete version of Sturm's theorem, \cite{Sturm93}, whose
proof is given in Section~\ref{s:proof_usg}.

\begin{thm}\label{main}
Assume the graph has uniform subexponential growth with respect to
an intrinsic metric with finite jump size. Then for any
$p\in[1,\infty]$
$$\sigma(L_p)=\sigma(L_2).$$
\end{thm}

\begin{Remark}(a) Other than \cite{Sturm93} we do not assume any type of bounded geometry or any types of lower bounds on the
curvature. For a discussion and examples see Section~\ref{s:growth}.

(b) Sometimes loops in the graph are modeled by non-vanishing
diagonal of $b$. However, the assumption that $b$ has zero diagonal
has no influence on our main results above as possible non-vanishing
diagonal terms  do not enter the operators. Such loops only have an
effect on $n$ and, thus, one would have to be careful if one chooses
$m=n$.\\  Clearly, we can also allow for potentials $c$ such that
$c/m$ is only bounded from below (as adding a positive constant
shifts the $\ell^{p}$ spectra  of  the operators simultaneously).
\end{Remark}

The following theorem shows that under an assumption on the measure
one spectral inclusion holds without any volume growth assumptions.

\begin{thm}\label{t:boundedmeasure} If $m$ is such that $\inf_{x\in X}m(x)>0$, then for any  $p\in [1,\infty]$
$$\si(L_{2})\subseteq \si(L_{p}).$$
\end{thm}

The proof of Theorem~\ref{t:boundedmeasure} is given in
Section~\ref{s:proof_upm}.

\section{Preliminaries}\label{s:preliminaries}
In this section we collect some results and facts that will be used
for the proof of Theorem~\ref{main}. Moreover, in the first
subsection we discuss the relation of uniform subexponential growth
and bounded geometry.

\subsection{Consequences of uniform subexponential growth}\label{s:growth}

\begin{lemma}\label{l:growth}Assume the graph has uniform subexponential growth. Then, for all  $\eps>0$  there is $C>0$ such that
\begin{itemize}
  \item [(a)]
        $m(x)\leq C e^{\eps d(x,y)}m(y)$ for all $x,y\in X$,
  \item [(b)]
        $\# B_{r}(x)\leq C e^{\eps r}$ for all $r\ge0$, where $\# B_{r}(x)$ denotes the number of vertices in $B_{r}(x)$.
  \item[(c)] $\sum_{y\in X} e^{-\eps d(x,y)}\leq C$ for all $x\in X$.
\end{itemize}
\end{lemma}
\begin{proof}
To prove (a) let $x,y\in X$. Using the uniform subexponential growth
assumption and $x\in B_{d(x,y)}(y)$ yields    $m(x)\leq
m(B_{d(x,y)}(y))\leq C e^{\eps d(x,y)}m(y)$. Turning to (b) let
$x\in X$ and $r\geq 0$. We obtain using (a) and the uniform
subexponential growth assumption {\begin{align*}
    \# B_{r}(x)= \sum_{y\in B_{r}(x)}{m(y)}/{m(y)}\leq Ce^{\eps r}\frac{m(B_{r}(x))}{m(x)}\le
    C^{2} e^{2\eps r}.
\end{align*}}
The final statement follows also by direct calculation using  (b)
(with $\eps_{1}$)
\begin{align*} \sum_{y\in X} e^{-\eps
d(x,y)}= \sum_{r=1}^{\infty}\sum_{y\in B_{r}(x)\setminus
B_{r-1}(x)}\hspace{-.5cm}e^{-\eps d(x,y)} \leq \sum_{r=1}^{\infty}
e^{-\eps  r}\# B_{r}(x)\leq C\sum_{r=1}^{\infty} e^{(\eps_{1}-\eps
)r}
\end{align*}
Hence choosing $\eps_{1}=\eps/2$ yields the statement.
\end{proof}

\begin{Remark}\label{r:locfin}(a) Lemma~\ref{l:growth}(b) implies finiteness of distance balls. On the other hand, finite jump size $s$ implies that for each vertex $x$  all neighbors of $x$ are contained in $B_s(x)$. Hence, graphs with uniform subexponential growth and finite jump size are locally finite.

(b)  Finiteness of distance balls has strong consequences on the
uniqueness of  selfadjoint extensions. In particular, by
\cite[Corollary~1]{HKMW} implies that $Q$ is the maximal form on
$\ell^{2}$ and that the restriction of $L_{2}$ to $C_{c}(X)$
(whenever $C_{c}(X)\subseteq D(L_{2})$) is essentially selfadjoint.
\end{Remark}

In the following  we discuss  examples to clarify the relation
between uniform subexponential growth and bounded geometry in the
discrete setting.

Recall that we speak of bounded geometry if $n/m$ is a bounded
function which is a natural adaption to the situation of weighted
graphs. In Example~\ref{ex:1} below, we show that there are uniform
subexponentially growing graphs with unbounded geometry. For
completeness we also give a example of bounded geometry and
exponential growth which is  certainly well-known.

\begin{eg}\label{ex:1}
(a) \emph{Uniform subexponential growth and unbounded geometry.} Let
$X=\N$, $m\equiv 1$, $c\equiv0$ and consider $b$ such that
$b(x,y)=0$ for $|x-y|\neq 1$, $b(x,x+1)=x$ for $x\in 4\N$ and
$b(x,x+1)=1$ otherwise. {Clearly} {$(n/m)(x)=n(x)=x+1$} for $x\in
4\N$ and, thus, $L_{p}$ is unbounded for all $p\in[1,\infty]$.
Moreover, let $d$ be the path metric induced by the edge weights
$w(x,x+1)=(n(x)\vee n(x+1))^{-\frac{1}{2}}$. We obtain that
{$d(x,y)\ge(|x-y|-3)/4\sqrt{2}$} for all $x,y\in X$. Hence,
$m(B_{r}(x))=\# B_{r}(x)\leq \sqrt{2}(8r+6)$ which implies uniform
subexponential growth.

(b) \emph{Exponential growth and bounded geometry.} Take a regular
tree,  $ c \equiv0$, set $b$ to be one on the edges and zero
otherwise and let $m\equiv1$. This graph has bounded geometry but is
clearly of exponential growth.
\end{eg}

It is apparent that the graph in Example~\ref{ex:1}(a) above has
bounded combinatorial vertex degree while the unbounded geometry is
induced by the edge weights. So, one might wonder whether one can
also present examples with unbounded combinatorial vertex degree
which is the criterion for unbounded geometry in the classical
setting. The proposition below shows that this is impossible under
the assumptions of uniform subexponential growth and finite jump
size. Recall that by the remark below Lemma~\ref{l:growth}
  we already know that the graph
must be locally finite.

The \emph{{combinatorial vertex degree}} $\deg$ is the function that
assigns to each vertex the number of neighbors, that is
$\deg(x)=\#\{y\in X\mid b(x,y)>0\}$, $x\in X$.

\begin{pro} If the graph has uniform subexponential growth with respect to a metric with finite jump size $s$, then the combinatorial vertex degree is bounded.
\end{pro}
\begin{proof} Suppose the graph has unbounded vertex degree, i.e., there is a sequence of vertices $(x_{n})$ such that $\deg(x_{n})\ge n^{2}$ for all $n\ge1$.
We show that there is a sequence of vertices $z_{n}$ such that
$m(B_{s}(z_{n}))/m(z_{n})$ is unbounded and thus the graph does not
have uniform subexponential growth.

If, for $n\ge1$, there is a neighbor $y_{n}$ of $x_{n}$ such that
$m(y_{n})\leq m(x_{n})/ \sqrt{\deg(x_{n})}$, then we  estimate using
$x_{n}\in B_{s}(y_{n})$
\begin{align*}
    \frac{m(B_{s}(y_{n}))}{m(y_{n})}\ge \frac{m(x_{n})}{m(y_{n})}\ge \sqrt{\deg(x_{n})}\ge n
\end{align*}
We set $z_{n}=y_{n}$ in this case. If, on the other hand, {$m(y)\geq
m(x_{n})/\sqrt{\deg(x_{n})}$} for all neighbors $y$ of $x_{n}$, then
\begin{align*}
    \frac{m(B_{s}(x_{n}))}{m(x_{n})}\ge \frac{1}{m(x_{n})} \deg(x_{n})\frac{m(x_{n})}{\sqrt{\deg(x_{n})}}\ge \sqrt{\deg(x_{n})}\ge n
\end{align*}
and set $z_{n}=x_{n}$ in this case. Hence, we have proven the claim.
\end{proof}

\begin{cor} Assume there is $D>0$ such that $b\leq D$
and $m\ge 1/D$. Then, uniform subexponential growth  with respect to
a metric with finite jump size implies bounded geometry.
\end{cor}
\begin{proof} One simply observes that $n/m\le D^{2}\deg$ and the statement follows from the proposition above.
\end{proof}

This means for the standard Laplacians $\Delta_{p}\ph(x)=\sum_{y\sim
x}(\ph(x)-\ph(y))$ on $\ell^{p}(X,1)$ and
$\Delta^{(n)}_{p}\ph(x)=\frac{1}{\deg(x)}\sum_{y\sim
x}(\ph(x)-\ph(y))$ on $\ell^{p}(X,\deg)$, that uniform
subexponential growth implies bounded geometry, both in the sense of
bounded $n/m$ being bounded and also in the sense of $\deg$ being
bounded.


\subsection{Lipschitz continuous functions}

We denote by $\Lip$  the real valued bounded Lipschitz continuous
functions with Lipschitz constant $\eps>0$, i.e.,
\begin{align*}
    \Lip:=\{\psi:X\to\R\mid \psi(x)-\psi(y)\leq \eps d(x,y),\,x,y\in X \}\cap \ell^{\infty}(X,m).
\end{align*}

\begin{lemma}\label{l:Lip}
Let {$\eps>0$} and let $s$ be the jump size of $d$. Then,  for all
$\psi\in \Lip$,
\begin{itemize}
  \item[(a)]  $e^{\psi}$ is a bounded Lipschitz continuous function,
  in particular, $e^{\psi}D(Q)=D(Q) $.
  \item[(b)]  $|1-e^{\psi(x)-\psi(y)}|\leq \eps e^{{\eps} s} d(x,y)$, for $x\sim y$.
  \item[(c)]  ${|(e^{-\psi(x)}-e^{-\psi(y)})(e^{\psi(x)}-e^{\psi(y)})|} \leq{2 \eps^2   e^{\eps s} d(x,y)^{2}}$, for $x\sim  y$.
\end{itemize}
\end{lemma}
\begin{proof}
The first statement of (a) follows from mean value theorem, that is
for any $x,y\in X$ we have
\begin{align*}
|e^{\psi(x)}-e^{\psi(y)}|\leq |\psi(x)-\psi(y)|e^{\|\psi\|_{\infty}}
\leq {\eps} d(x,y)e^{\|\psi\|_{\infty}}.
\end{align*}
Now, $e^{\psi}D(Q)\subseteq D(Q) $ is a consequence of
\cite[Lemma~3.5]{GHM}. The other inclusion follows since $e^{-\psi}$
is also bounded and Lipschitz continuous. Similarly, we get (b)
using the Taylor expansion of the exponential function
\begin{align*}
|1-e^{\psi(x)-\psi(y)}|=
\sum_{k\ge1}\frac{(\psi(x)-\psi(y))^{k}}{k!}
 \leq\eps d(x,y)\sum_{k\ge1}\frac{( {\eps} s)^{k-1}}{k!}
\leq \eps d(x,y) e^{ {\eps}s},
\end{align*}
and, similarly, using $
|{(e^{-\psi(x)}-e^{-\psi(y)})(e^{\psi(x)}-e^{\psi(y)})}|
=2\sum_{k\in2\N}\frac{(\psi(x)-\psi(y))^{k}}{k!}$  we get (c).
\end{proof}

\subsection{Kernels}
Let $A:D(A)\subseteq\ell^p\rightarrow \ell^q$, $p,q\in[1,\infty]$ be
a densely defined linear operator. We denote by $\|A\|_{p,q}$ the
operator norm of $A,$ i.e. $$\|A\|_{p,q}=\sup_{f\in D(A),\
\|f\|_p=1}\|Af\|_q.$$ Note that any such  operator
$A:D(A)\subseteq\ell^p\rightarrow \ell^q$, $p<\infty$, with
$C_{c}(X)\subseteq D(A)$ admits a kernel $k_{A}:X\times X\to\C$ such
that
\begin{align*}
Af(x)=\sum_{x\in X}k_{A}(x,y)f(y)m(y)
\end{align*}
for all $f\in D(A)$, $x\in X$, which can be obtained by
\begin{align*}
    k_{A}(x,y)=\frac{1}{m(x)m(y)}\langle A 1_{y},1_{x}\rangle,
\end{align*}
where $1_{v}(w)=1$ if {$w=v$} and $1_{v}(w)=0$ otherwise.

We recall the following well known lemma which shows that the
operator norm of $A:\ell^p\rightarrow \ell^q$ can be estimated by
its integral kernel.
\begin{lemma}\label{l:kernels} Let $p\in {[1,\infty)}$ and let $A$ be a densely defined linear operator with $C_{c}(X)\subseteq D(A)\subseteq \ell^{p}$. Then,
\begin{itemize}
  \item [(a)] $\|A\|_{p,q}\leq
\left(\sum_y\|k_A(\cdot,y)\|_q^{p^{*}}m(y)\right)^{\frac{1}{p^{*}}}$
for $q<\infty$,
  \item [(b)] $\|A\|_{p,\infty}\leq\sup_{x}\|k_{A}(x,\cdot)\|_{p^{*}}$ and equality holds if $p=1$.
\end{itemize}
\end{lemma}
\begin{proof} (a) follows from the fact $\|f\|_{p}=\sup_{g\in\ell^{p^{*}},\,{\|g\|}_{p^{*}}=1}\langle g,f\rangle$ and
twofold application of H\"older inequality. The first part of (b)
follows simply from  H\"older inequality. For the second part note
 that $\|A\|_{1,\infty}\ge\sup_{x,y\in X}|A1_{{y}}(x)|/m(y)$
\end{proof}


\subsection{Heat kernel estimates}

We denote the kernel of the semigroup $T_{t}$, $t\ge0$ by $p_{t}$.
As the semigroups are consistent on $\ell^{p}$, $p\in[1,\infty]$,
i.e., they agree on {their} common domains, the kernel $p_{t}$ does
not depend on $p$.

The following heat kernel estimate will be the key to
$p$-independence of spectra of $L_p$. It is proven in  \cite{FOLZ},
based on \cite{Davies93},  for locally finite graphs and $c\equiv0$.
However, on the one hand local finiteness is not used in \cite{FOLZ}
for this result and, on the other hand, the remark below
Lemma~\ref{l:growth} shows that we are in the local finite situation
anyway whenever we assume uniform subexponential growth. We conclude
the statement for $c\ge0$ by a Feynman-Kac formula.

\begin{lemma}\label{l:heatkernel1}We have for all $t\geq0$ and $x,y\in X$
\begin{align*}
p_t(x,y)\leq(m(x)m(y))^{-\frac{1}{2}}e^{-d(x,y)
\log\frac{d(x,y)}{2et}}.
\end{align*}
\end{lemma}
\begin{proof} Denote the semigroup of the graph $(b,0)$ by $T^{(0)}_{t}$ and the kernel by $p_{t}^{(0)}$ and correspondingly for $(b,c)$ by $T_{t}$ and $p_{t}$.

For $c\equiv0$ the estimate  is found in \cite[Theorem~2.1]{FOLZ}
for $p_{t}^{(0)}$.  Now, by a Feynman-Kac formula, see e.g.
\cite{demuth,GKS}, we have
\begin{align*}
p_{t}(x,y)= T_{t}\de_{y}(x)=
\mathbb{E}_{x}\big[e^{-\int_{0}^{t}\frac{c}{m}(\mathbb{X}_{s})ds}
\de_{y}(\mathbb{X}_{t})\big]
\leq\mathbb{E}_{x}\big[\de_{y}(\mathbb{X}_{t})\big]
=T^{(0)}_{t}\de_{y}(x)=p_{t}^{(0)}(x,y),
\end{align*}
where $\de_{y}=1_{y}/m(y)$. This proves the claim.
\end{proof}

By  basic calculus, we obtain the following heat kernel estimate.
\begin{lemma}\label{l:heatkernel2}
For all $\beta>0$ there exists a constant $C(\beta)$ such that for
all $t\geq0$, $x,y\in X$
\begin{align*}
p_t(x,y)\leq (m(x)m(y))^{-\frac{1}{2}}e^{-\beta d(x,y)+C(\beta)t}.
\end{align*}
\end{lemma}
\begin{proof}
Let $\beta>0$ and $x>0$ let $f(x)=-x\log(x/2e)+\beta x$. Direct
calculation shows that the function $f$ assumes its maximum  on the
domain $(0,\infty)$ at the point $x_0=2e^{\beta}$. In particular,
setting $C(\beta)=2e^{\beta}$ yields
\begin{align*}
-\frac{d(x,y)}{t}\log \frac{d(x,y)}{2et}\leq -\beta
\frac{d(x,y)}{t}+C(\beta).
\end{align*}
for all $t>0$ and $x,y\in X$.
\end{proof}


\section{Proof for uniform subexponential growth}\label{s:proof_usg}

In this section we prove  Theorem~\ref{main} following the strategy
of \cite{Sturm93}. The proof is divided into several lemmas and as
always we assume that $d$ is an intrinsic metric with finite jump
size $s$.

\begin{lemma}\label{l:lemma1}For every compact set $K\subseteq \rho(L_2)$ there is $\eps>0$ and $C<\infty$ such that for all $z\in K$ and all $\psi\in \Lip$
\begin{align*}
\|e^{-\psi}(L_{2}-z)^{-1}e^{\psi}\|_{2,2}\leq C.
\end{align*}
\end{lemma}
\begin{proof}
Let $\eps>0$ and $\psi\in \Lip$. By Lemma~\ref{l:Lip} we have
$e^{-\psi}D(Q)=e^{\psi}D(Q)=D(Q)$. Let $Q_{\psi}$ be the (not
necessarily symmetric) form with domain $D(Q_{\psi})=D(Q)$ acting as
\begin{align*}
Q_{\psi}(f,g):=    Q(e^{-\psi}f,e^{\psi}g)-Q(f,g). \end{align*}
Application of Leibniz rule yields, for $f\in D(Q_{\psi})$,
\begin{align*}
|Q_{\psi}(f,f)|=&\frac{1}{2}\sum_{x,y\in X} b(x,y)
|f(y)|^{2}(e^{-\psi(x)}-e^{-\psi(y)}) (e^{\psi(x)}-e^{\psi(y)})
\\
&+\frac{1}{2}\sum_{x,y\in X} b(x,y)\ov{f(y)} {(f(x)-f(y))}(1-e^{\psi(y)-\psi(x)})\\
&+\frac{1}{2}\sum_{x,y\in X} b(x,y)f(y)
(1-e^{\psi(x)-\psi(y)})\overline{(f(x)-f(y))}.
\end{align*}
Applying Cauchy-Schwarz inequality, Lemma~\ref{l:Lip}(b) and (c) and
the intrinsic metric property, gives
\begin{align*}
\ldots    \leq &\eps^{2}C\sum_{x,y\in X}|f(x)|^{2}b(x,y)d(x,y)^{2}
+2C\eps
\Big(\sum_{x,y\in X}|f(x)|^{2} b(x,y)d(x,y)^{2}\Big)^{\frac{1}{2}}Q(f)^{\frac{1}{2}}\\
&\leq C\eps^{2}\|f\|^{2}_{2}+2C\eps \|f\|_{2}Q(f)^{\frac{1}{2}}.
\end{align*}
Hence, the basic inequality $2ab\leq(1/\de) a^{2}+\de b^{2}$ for
$\de>0$ and $a,b\ge0$ (applied with  $a=C\eps\|f\|^{2}_{2} $ and
$b=Q(f)^{\frac{1}{2}}$) yields
\begin{align*}
|Q_{\psi}(f,f)|\leq  C \eps^{2}(1+\tfrac{1}{\de})\|f\|^{2}_{2}+\de
Q(f).
\end{align*}
This shows that $Q_{\psi}$ is $Q$ bounded with bound $0$. According
to \cite[Theorem~VI.3.9]{Kato} this implies that the form
$Q_{\psi}+Q$ is closed and sectorial. It can be checked directly
that the corresponding operator is $e^{\psi}L_{2}e^{-\psi}$ with
domain $D_{\psi}=e^{\psi}D(L_{2})$. Moreover, for $K\subseteq
\rho(L_{2})$ compact, we can choose $\eps,\de>0$ that $2\big\|\big
(C(1+1/\de)\eps^{2}e^{2s}+\de
L_{2}\big)\cdot(L_{2}-z)^{-1}\big\|_{2,2}<1$ for all $z\in K$ since
$C$ is a universal constant. Therefore, again by
\cite[Theorem~VI.3.9]{Kato}  this implies existence of $C=C(K,\eps)$
such that
$$\|e^{-\psi}(L_{2}-z)^{-1}e^{\psi}\|_{2,2}
=\|(e^{\psi}L_{2}e^{-\psi}-z)^{-1}\|_{2,2}\leq C$$ for all $z\in K$
and $\psi\in \Lip$.
\end{proof}

Let us recall some well known facts about consistency of semigroups
and resolvents. By \cite[Theorem~1.4.1]{Davies89} the semigroups
$T_{t}$ are consistent on $\ell^{p}$. By the spectral theorem the
Laplace transform for the resolvent $G_{z}=(L_{2}-z)^{-1}$
\begin{align*}
G_{z}f=\int_{0}^{\infty}e^{z t}T_{t}f dt,
\end{align*}
holds for $f$ in $\ell^{2}$ and $z\in \{w\in\C\mid\Re w<0\}$ (the
open left half plane). By density and duality arguments, this
formula extends to $f$ in $\ell^{p}$, $p\in[1,\infty)$, in the
strong sense and to $p=\infty$ in the weak sense. This shows that
the resolvents $(L_{p}-z)^{-1}$ are consistent on $\ell^{p}$ for
$z\in \{w\in\C\mid\Re w<0\}$.

We denote by $g_{\al}$ the kernel of the resolvent
$G_{\al}=(L_{p}-\al)^{-1}$ which is independent of $p\in[1,\infty]$
for $\al<0$.

\begin{lemma}\label{l:lemma2}Assume the graph has uniform subexponential volume growth. For any $\eps>0$ there exists $\al<0$  and $C<\infty$ such that
\begin{itemize}
  \item[(a)] $|g_{\al}(x,y)|\leq C(m(x)m(y))^{-\frac{1}{2}}e^{-\eps d(x,y)}$
for all $x,y\in X$,
  \item [(b)] $\|e^{\psi}G_{\al} e^{-\psi}m^{\frac{1}{2}}\|_{1,2}\leq C$ for all $\psi\in \Lip $,
  \item [(c)] $\|m^{\frac{1}{2}}e^{\psi}G_{\al} e^{-\psi}\|_{2,\infty}\leq C$ for all $\psi\in \Lip$.
\end{itemize}
\end{lemma}
\begin{proof}
(a) By the Laplace transform of the resolvent and
Lemma~\ref{l:heatkernel2}, we get
\begin{align*}
g_{\al}(x,y)=\int_{0}^{\infty}e^{\al t}p_{t}(x,y)dt\leq
(m(x)m(y))^{-\frac{1}{2}}e^{-\eps d(x,y)}\int_{0}^{\infty}e^{(\al
+C)t}dt,
\end{align*}
which yields the statement for $\al< -C$.\\
(b) By Lemma~\ref{l:kernels}(a), $\psi\in \Lip$, part (a) above
(with $\eps_{1}\ge 2\eps$)  and Lemma~\ref{l:growth}(c)
\begin{align*}
 \|e^{\psi}G_{\al} e^{-\psi}m^{\frac{1}{2}}\|_{1,2}^{2}&\leq\sup_{y\in X}\|g_{\al}(\cdot,y)e^{\psi(\cdot)-\psi(y)}m(y)^{\frac{1}{2}}\|_{2}^{2}\\
&\leq C\sup_{y\in X}\sum_{x\in X} e^{(2\eps-2\eps_1) d(x,y)}<\infty.
\end{align*}
The proof of (c) works similarly using Lemma~\ref{l:kernels}(b).
\end{proof}

\begin{lemma}\label{l:lemma3} Assume the graph has uniform subexponential volume growth. Then, $(L_{2}-z)^{-2}$ extends to a bounded operator on $\ell^{p}$ for all $z\in\rho(L_{2})$ and $p\in[1,\infty]$. Moreover, for all compact $K\subseteq \rho(L_{2})$
there is $C<\infty$ such that for all $z\in K$ and $p\in[1,\infty]$
$$\|(L_{2}-z)^{-2}\|_{p,p}\leq C.$$
\end{lemma}
\begin{proof}
For $z\in \rho(L_{2})$ denote by $g_{z}^{(2)}$ the kernel of the
squared resolvent $(G_{z})^{2}=(L_{2}-z)^{-2}$. Applying the
resolvent identity twice yields \begin{align*} (G_{z})^{2}=
(G_{\al}+(z-\al)G_{\al}G_{z})(G_{\al}+(z-\al)G_{z}G_{\al})=
G_{\al}(I+(z-\al)G_{z})^{2}G_{\al},
\end{align*}
for all $\al<0$. Therefore,
\begin{align*}
m^{\frac{1}{2}}e^{\psi}(G_{z})^{2}e^{-\psi}m^{\frac{1}{2}}=
\Big(m^{\frac{1}{2}}e^{\psi}G_{\al}e^{-\psi}\Big)
\Big(I+(z-\al)e^{\psi/2} G_{z}e^{-\psi/2}\Big)^{2}
\Big(e^{\psi}G_{\al}e^{-\psi}m^{\frac{1}{2}}\Big),
\end{align*}
for all $\eps>0$ and $\psi\in\Lip$. Taking the norm
$\|\cdot\|_{1,\infty}$ and factorizing
{$\|\ldots\|_{1,\infty}\leq\|(\ldots)\|_{2,\infty}
\|(\ldots)\|_{2,2}\|(\ldots)\|_{1,2}$} {yields} that
$U:=m^{\frac{1}{2}}e^{\psi}G_{z}^{2}e^{-\psi}m^{\frac{1}{2}}$ is a
bounded operator $\ell^{1}\to\ell^{\infty}$ by Lemma~\ref{l:lemma1}
and Lemma~\ref{l:lemma2} with appropriate choice of $\al<0$,
$\eps>0$ and all $\psi\in \Lip$. Hence, the operator $U$ admits a
kernel
$k_{U}(x,y)=(m(x)m(y))^{\frac{1}{2}}e^{\psi(x)-\psi(y)}g_{z}^{(2)}(x,y)$,
$x,y\in X$, and we conclude from Lemma~\ref{l:kernels}(b) that
\begin{align*}
|g_{z}^{(2)}(x,y)|\le C
(m(x)m(y))^{-\frac{1}{2}}e^{\psi(y)-\psi(x)}.
\end{align*}
{For chosen $\eps >0$ and any} fixed $x,y\in X$ let $\psi:u\mapsto
\eps(d(u,y)\wedge d(x,y))$  and we obtain from
Lemma~\ref{l:growth}(a) (with $\eps$)
\begin{align*}
|g_{z}^{(2)}(x,y)|\leq C    (m(x)m(y))^{-\frac{1}{2}}e^{-\eps
d(x,y)}\leq C    m(x)^{-1}e^{-\frac{\eps}{2} d(x,y)}.
\end{align*}
Thus, using Lemma~\ref{l:kernels}(a) and Lemma~\ref{l:growth}(c), we
obtain
\begin{align*}
\|G_{z}^{2}\|_{1,1}\leq \sup_{y\in X}\sum_{x\in
X}|g_{z}^{(2)}(x,y)|m(x)\leq C \sup_{y\in X}\sum_{x\in
X}e^{-\frac{\eps }{2}d(x,y)} <\infty.
\end{align*}
As $G_{z}^{2}$ is bounded for $p=1$ and $p=2$, it follows from the
Riesz-Thorin interpolation theorem that it is bounded for
$p\in[1,2]$ and by duality for $p\in[1,\infty]$.
\end{proof}

\begin{lemma}\label{l:lemma4} If $\si(L_{p})\subseteq[0,\infty)$ for all $p\in [1,\infty]$, then $\si(L_{2})\subseteq\si(L_{p})$.
\end{lemma}
\begin{proof}
The operators $(L_{p}-z)^{-1}$ and $(L_{p*}-z)^{-1}$ are consistent
for $z\in \{w\in\C\mid\Re w<0\}\subseteq
\rho(L_{p})=\rho(L_{p^{*}})$ by the discussion above
Lemma~\ref{l:lemma2} for all $p\in[1,\infty]$. By the assumption
$\si(L_{p})\subseteq [0,\infty)$ the resolvent sets are connected
which yields by \cite[Corollary~1.4]{HempelVoigt87} that
$(L_{p}-z)^{-1}$ and $(L_{q}-z)^{-1}$ are consistent for $z\in
\rho(L_{p})\cap\rho(L_{q})$ for $p,q\in[1,\infty]$. Moreover, by the
standard theory \cite[Lemma~8.1.3]{Davies07} $(L_{p}-z)^{-1}$ and
$(L_{p*}-z)^{-1}$ are analytic on $\rho(L_{p})=\rho(L_{p^{*}})$. By
the Riesz-Thorin theorem these resolvents can be consistently
extended to analytic $\ell^2$-bounded operators, see
\cite[Lemma~1.4.8]{Davies07}. That is, as a $\ell^2$-bounded
operator-valued function $(L_{p}-z)^{-1}$ is analytic on $\rho(L_p)$
which is consistent with $(L_2-z)^{-1}$ on $\rho(L_p)\cap\rho(L_2).$
Note that, $(L_2-z)^{-1}$ is analytic on $\rho(L_2)$ which is also
the maximal domain of analyticity. Thus, the statement follows by
unique continuation.
\end{proof}

\begin{proof}[Proof of Theorem~\ref{main}] We start by showing $\si(L_{p})\subseteq \si(L_{2})$.
For the kernel $g_{z}^{(2)}$ of $(L_2-z)^{-2}$ and fixed $x,y\in X$
the function $\rho(L_{2})\to\C$, $z\mapsto g_{z}^{(2)}(x,y)$ is
analytic. By Lemma~\ref{l:lemma3} we know that for any compact
$K\subseteq \rho(L_{2})$ the operators $(L_2-z)^{-2}$, $z\in K$ are
bounded on $\ell^{p}$, $p\in[1,\infty]$. Therefore $(L_2-z)^{-2}$ is
analytic as a family of $\ell^{p}$-bounded operators for
$z\in\rho(L_{2})$. On the other hand, $(L_{p}-z)^{-2}$ is analytic
as a family of $\ell^{p}$-bounded operators with domain of
analyticity $\rho(L_{p})$ by \cite[Lemma~3.2]{HempelVoigt86}. Since
$(L_2-z)^{-2}$ and $(L_{p}-z)^{-2}$ agree on $\{w\in\C\mid\Re
w<0\}$, by unique continuation they agree as analytic $\ell^p$
operator-valued functions on $\rho(L_{2})$. As the domain of
analyticity of $(L_{p}-z)^{-2}$ is $\rho(L_{p})$, this implies
$\rho(L_{2})\subseteq \rho(L_{p})$.

On the other hand, since  $\si(L_{p})\subseteq
\si(L_{2})\subseteq[0,\infty)$  the statement $\si(L_{2})\subseteq
\si(L_{p})$ follows from Lemma~\ref{l:lemma4}.
\end{proof}


\section{Proof for uniformly positive measures}\label{s:proof_upm}
In this section we consider measures that are uniformly bounded from
below by a positive constant and prove
Theorem~\ref{t:boundedmeasure}. We notice that $\inf_{x\in X}m(x)>0$
implies
\begin{align*}
    \ell^{p}\subseteq \ell^{q},\quad 1\le p\leq q\leq \infty.
\end{align*}
Moreover, by \cite[Theorem~5]{KL} we know the domains of the
generators $L_{p}$ in this case explicitly, namely
\begin{align*}
    D(L_{p}):=\{f\in\ell^{p}\mid \ow L f\in \ell^{p}\}.
\end{align*}
In particular, this gives
\begin{align*}
    D(L_{p})\subseteq D(L_{q}),\quad 1\le p\leq q\leq \infty,
\end{align*}
where $\ow L$ was defined in Section~\ref{s:operators}. Furthermore,
it can be checked directly that $C_{c}(X)\subseteq D(L_{p})$.

\begin{lemma}\label{l:resolvents} Assume
$\inf_{x\in X}m(x)>0$. Then, for all $1\leq p\leq q\leq\infty$ and
all $z\in \rho(L_{p})\cap\rho(L_{q})$ the resolvents
$(L_{p}-z)^{-1}$ and $(L_{q}-z)^{-1}$ are consistent on
$\ell^{p}=\ell^{p}\cap\ell^{q}$.
\end{lemma}
\begin{proof} Let $1\leq p< q\leq \infty$ and $z\in \rho(L_{p})\cap\rho(L_{q})$. As $D(L_{p})\subseteq D(L_{q})$ and $L_{q}=L_{p}$ on $D(L_{p})$, we have for all $f\in\ell^{p}\subseteq\ell^{q}$
\begin{align*}
    (L_{q}-z)(L_{p}-z)^{-1}f=    (L_{p}-z)(L_{p}-z)^{-1}f=f
\end{align*}
Hence,  $(L_{p}-z)^{-1}$ and $(L_{q}-z)^{-1}$ are consistent on
$\ell^{p}=\ell^{p}\cap\ell^{q}$.
\end{proof}

\begin{proof} [Proof of Theorem~\ref{t:boundedmeasure}]
Let $p\in[1,2]$. By the lemma above the resolvents $(L_{p}-z)^{-1}$
and $(L_{p^{*}}-z)^{-1}$ are consistent for
$z\in\rho(L_{p})=\rho(L_{p^{*}})$ on $\ell^{p}$. By the Riesz-Thorin
interpolation theorem $(L_{p^{*}}-z)^{-1}$ is bounded on $\ell^{2}$.
We will show that $(L_{p^{*}}-z)^{-1}$  is an inverse of $(L_{2}-z)$
for $z\in\rho(L_p)=\rho(L_{p*})$. So, {let $z\in\rho(L_{p*})$.} As
$D(L_{2})\subseteq D(L_{p^{*}})$, {$\ell^{2}\subseteq\ell^{p^*}$}
and $L_{2}$, $L_{p^{*}}$ are restrictions of $\ow L$ we have for
$f\in D(L_{2})$
\begin{align*}
(L_{p^{*}}-z)^{-1}(L_{2}-z)f=(L_{p^{*}}-z)^{-1}(L_{p^{*}}-z)f=f.
\end{align*}
Secondly, let $f\in \ell^{2}$ and $(f_{n})$ be such that
$f_n\in\ell^{p}$ and $f_n\to f$ in $\ell^{2}$. As
$(L_{p^{*}}-z)^{-1}$ is $\ell^{2}$-bounded,
$(L_{p^{*}}-z)^{-1}f_{n}\to (L_{p^{*}}-z)^{-1}f$, $n\to\infty$, in
$\ell^{2}$. By the lemma above
$(L_{p^{*}}-z)^{-1}f_{n}=(L_{p}-z)^{-1}f_{n}\in D(L_{2})$, and,
thus,
\begin{align*}
(L_{2}-z)(L_{p^{*}}-z)^{-1}f_{n}=(L_{p^{*}}-z)(L_{p^{*}}-z)^{-1}f_{n}=f_{n}\to
f,\quad n\to\infty,
\end{align*}
in $\ell^{2}$. Since, $L_{2}$ is closed we infer
$(L_{p^{*}}-z)^{-1}f\in D(L_{2})$ and $(L_{2}-z)(L_{p^{*}}-z)^{-1}f=
f$. Hence, $(L_{p^{*}}-z)^{-1}$ is an inverse of $(L_{2}-z)$ and,
thus,  $z\in \rho(L_{2})$.
\end{proof}

\begin{Remark}The abstract reason behind Theorem~\ref{t:boundedmeasure} is that the semigroup $e^{-tL}$ is ultracontractive, i.e. a bounded operator from $\ell^{2}$ to $\ell^{\infty}$, which is a consequence of
 $e^{-tL}$ being a contraction on $\ell^{\infty}$ (as $Q$ is a Dirichlet form) and the uniform lower bound on the measure. Knowing this one can deduce by duality and interpolation that $e^{-tL}$ is a bounded operator from $\ell^{p}$ to $\ell^{q}$, $p\leq q$, (cf. \cite[Proof of Theorem~B.1.1]{Simon82}) and employ the proof of \cite[Proposition~2.1]{HempelVoigt86} or \cite[Proposition~3.1]{HempelVoigt87}
\end{Remark}


\section{Spectral Properties of Normalized Laplacians}\label{s:normalized}
In this section, we consider normalized Laplace operators that is we
assume $m=n$, $c\equiv0$ where $n(x)=\sum_{y\in X}b(x,y)$, $x\in X$.

\subsection{Symmetries of the spectrum}
Recall that  a graph is called \emph{bipartite} if the vertex set
$X$ can be divided into two disjoint subsets {$X_1$} and $X_2$ such
that every edge connects a vertex in $X_1$ to a vertex in $X_2$,
i.e., $b(x,y)>0$ for a pair $(x,y)\in X\times X$  implies $(x,y)\in
(X_1\times X_2)\cup(X_2\times X_1)$.

\begin{thm}\label{thm11}Assume $m=n$, {$c\equiv0$} and let {$p\in[1,\infty]$}. Then, $\sigma(L_p)$
is included in $\{z\in\C\mid|z-1|\leq1\}$ and is symmetric with
respect to $\R=\{z\in\C\mid\Im z=0\}$, i.e., $\lm\in\sigma(L_p)$ if
and only if $\overline\lm\in\sigma(L_{p})$. Moreover, if the graph
is bipartite, then $\sigma(L_p)$ is symmetric with respect to the
line $\{z\in\C: \Re  z=1\},$ i.e., $\lambda\in\sigma(L_p)$ if and
only if $(2-\lambda)\in\sigma(L_p).$
\end{thm}
The proof of the second part of the theorem is based on the
following lemma, \cite[Lemma~1.2.13]{Davies07}.

\begin{lemma}\label{eigenvalue}
Let $A: \mathcal{B}\rightarrow \mathcal{B}$ be a bounded operator on
a Banach space $\mathcal{B}$. Then $\lambda\in\sigma(A)$ if and only
if at least one of the following occurs:
\begin{itemize}
\item [(i)] $\lambda$ is an eigenvalue of A. \item[(ii)] $\lambda$
is an eigenvalue of  $A^{\ast}$, where $A^{\ast}$ is the dual
operator of A.
\item[(iii)] There exists a sequence $(f_n)$  in $
\mathcal{B}$ with $\|f_n\|=1$ and
$\lim\limits_{n\to\infty}\|Af_n-\lambda f_n\|=0.$
\end{itemize}
\end{lemma}

\begin{proof}[{Proof of Theorem~\ref{thm11}}]
Let $p\in[1,\infty]$. As $n/m=1$ the operator $L_{p}$ is bounded on
$\ell^{p}$ by  \cite[Theorem~11]{KL2} (or \cite[Theorem~9.3]{HKLW})
with bound $2$. Moreover, the operator $L_{p}$ can be represented as
$L_{p}=I-P_{p}$ where $P_{p}$ is the transition matrix acting as
$P_{p}f(x)=\frac{1}{n(x)}\sum_{y}b(x,y)f(y)$. {Direct calculation
shows that $\|P_{p}\|_{p,p}\leq 1.$} As the spectral radius is
smaller than the norm, $\si(P_{p})\subseteq \{z\in\C\mid|z|\leq1\}$.
Now, $I$ is  a spectral shift by $1$, so the first statement
follows.

{The first symmetry statement follows from the fact that the
integral kernel of $L_p$ is real valued.}

Let now $X_{1}$ and $X_{2}$ {be a} bipartite partition of $X$. For a
function $f:X\to\C$, we denote
\begin{align*}
    \ow f= 1_{X_{1}}f-1_{X_{2}}f,
\end{align*}
where $1_{W}$ is the characteristic function of $W\subseteq X$. We
separate the proof into three cases according to the preceding
lemma.

Case 1:  $\lm$ is an eigenvalue of $L_{p}$  with eigenfunction $f$.
By direct calculation it can be checked that $\ow f$ is an
eigenfunction of $L_{p}$ for the eigenvalue $2-\lm$.

Case 2: $\lm$ is an eigenvalue of $L_{p}^{*}$. By the same argument
as in Case 1 we get that $2-\lm$ is a eigenvalue of $L_{p^{*}}$.
Therefore, by $\sigma(L_p^{*})=\sigma(L_p)$, we conclude the result.

Case 3:  $\lm$ is such that there is $f_{n}$ with $\|f_{n}\|=1$,
$n\ge 1$ and $\lim\|(L_{p}-\lm)f_{n}\|_{p}=0$. It is easy to see
that $\|\ow f_{n}\|_{p}=\|f_{n}\|_{p}$ and by direct calculation it
follows $\|(L_{p}-(2-\lm))\ow f_{n}\|_{p}=\|(L_{p}-\lm)f_{n}\|_{p}$.
Thus, the statement $(2-\lm)\in\si(L_{p})$  follows by
Lemma~\ref{eigenvalue}(iii).
\end{proof}

\subsection{Cheeger constants}

Define the \emph{Cheeger constant} $\al\ge0$  to be the {maximal}
$\be\ge0$ such that for all finite $W\subseteq X$
\begin{align*}
    \beta n(W)\leq |\partial W|,
\end{align*}
where {$|\partial W|:=\sum_{(x,y)\in W\times {(X\setminus
W)}}b(x,y)$}.  For recent developments concerning Cheeger constants
see also \cite{BHJ, BKW}.

\begin{pro}\label{p:alpha} Assume $m=n$ and $c\equiv0$. Then,  $\inf\si(L_{1})=\inf\si(L_{2})$ if and only if $\al=0$.
\end{pro}
\begin{proof} As $L_{\infty}$ is bounded it is easy to see that the constant
functions are eigenfunctions to the eigenvalue $0$.  Hence,
$\inf\si(L_{1})=\inf\si(L_{\infty})=0$. Now, $0=\inf\si(L_{2})$ is
equivalent to $\al=0$ by a Cheeger inequality, \cite{KL2} (cf.
\cite{DK} and \cite{DoKa88}).
\end{proof}

\begin{Remark}
The proposition can also be  obtained as a consequence of
\cite[Theorem~3.1]{Takeda}. The considerations therein show that one
direction of the result is still valid in certain situations
involving unbounded operators using the Cheeger constant defined in
\cite{BKW}. However, in this case it is usually hard to determine
whether the constant functions are in the domain of $L_{\infty}$.
Nevertheless, this is the case if the graph is stochastically
complete, (see \cite{KL} for characterizations of this case).
\end{Remark}

\subsection{Spectral independence and superexponential growth}

This previous proposition shows that $\alpha=0$ is a necessary
condition for the $p$-independence of the spectrum for
$p\in[1,\infty]$. However, the next proposition shows that if we
exclude $p=\{1,\infty\}$, then we can have $p$-independence of the
spectrum even if $\alpha>0$ or if the subexponential volume growth
condition is not satisfied.

Define the \emph{Cheeger constant at infinity} $\al_{\infty}\ge0$ to
be the maximum of all $\be\ge0$  such that for some finite
$K\subseteq X$ and all finite $W\subseteq X\setminus K$ and $\beta
n(W)\leq |\partial W|.$

One can easily check that under the assumption of $m(X)=\infty,$
$\alpha>0$ if and only if $\alpha_\infty>0$: The inequality
$\al\le\al_{\infty}$ is obvious. If $\al=0$ but $\al_{\infty}>0$,
the Cheeger estimates, see e.g. \cite{Fujiwara96,Keller10, KL2}
imply $0\in\si(L_{2})$, but $0\not\in\si_{\mathrm{ess}}(L_{2})$.
This, however, means that $0$ is an eigenvalue which is impossible
as {the constant functions are not in $\ell^{2}$ due to
$m(X)=\infty$,} cf. \cite[Theorem~4.1]{GHKLW}. Hence, the next
theorem says that we have $p$-independence for $p\in(1,\infty)$ even
though $\alpha>0$.

\begin{thm}\label{t:superexponential} Assume $m=n$ and $\al_{\infty}=1$. Then, $\si(L_{p})=\si(L_{2})$ for all $p\in(1,\infty)$ and the graph has superexponential volume growth with respect to the natural graph metric, i.e., $\lim_{r\to\infty}\frac{1}{r}\log n(B_{r}(x))=\infty$, for all $x\in X$.
\end{thm}
\begin{proof}It was shown in \cite{KL2}  (cf.  \cite{Keller10, Fujiwara96} for the unweighted case) that $\al_{\infty}=1$ implies
$\sigma_{\mathrm{ess}}(L_2)=\{1\}$ which is equivalent to $I-L_2$
being a compact operator. Since $I-L_p$ is a consistent family of
bounded operators for $p\in[1,\infty]$ it follows from a {well-known
result} by Krasnosel'skii and Persson (see
\cite[Theorem~4.2.14]{Davies07}) that $I-L_p$ is compact for all
$p\in [2,\infty)$. Using a theorem of Schauder (see for instance
\cite[Theorem~4.2.13]{Davies07}), we conclude that $I-L_p$ is
compact for all $p\in (1,\infty)$. Now it follows from
\cite[Theorem~4.2.15]{Davies07}  that the spectrum of
 $L_p$ is $p$-independent for all $p\in (1,\infty)$. The second part of the proposition follows directly from \cite[Theorem~1]{Fujiwaragrowth} (for the weighted case combine \cite[Theorem~19]{KL2} and \cite[Theorem~4.1]{HKW}).
\end{proof}

\begin{rem}
(a)  For examples satisfying the assumptions of the theorem above
see \cite{Keller10, Fujiwara96}.

(b) Under the assumption  $m=n$ and $\alpha>0$ one can show that
$I-L_1$ and $I-L_\infty$ are not compact operators. Assume the
contrary, i.e., $I-L_1$ and $I-L_\infty$ are compact, then
\cite[Theorem 4.2.15]{Davies07} implies that the spectrum is
$p$-independent for all $p\in[1,\infty]$. This, however, contradicts
Proposition \ref{p:alpha}.
\end{rem}

\section{Tessellations and curvature}\label{s:tessellations}

For this final section, we restrict  our attention to planar
tessellations and relate curvature bounds to volume growth and
$p$-independence. We show analogue statements to Proposition~1 and~2
of \cite{Sturm93} and discuss the case of uniformly unbounded
negative curvature.

We consider graphs such that $b$ takes values in $\{0,1\}$ and
$c\equiv0$. The two prominent choices for $m$ are either $m=n$ or
$m\equiv 1$. For $m=n$ the natural graph metric $d_{n}$ is an
intrinsic metric and for $m\equiv 1$ the path metric $d_{1}$ given
by
$$d_{1}(x,y)=\inf_{x=x_{0}\sim\ldots\sim x_{n}=y}\sum_{i=1}^{n}(n(x_{i-1})\vee n(x_{i}))^{-\frac{1}{2}}, \quad x,y\in X,$$ is an intrinsic metric. Both metrics have the jump size at most $1$, in particular, both metrics have finite jump size.
We denote the Laplacian with respect to $m=n$ by $\Delta^{(n)}_{p}$
and with respect to $m\equiv 1$ by $\Delta_{p}$, {$p\in[1,\infty]$}.

Note that by Theorem~\ref{t:boundedmeasure} we have for all
$p\in[1,\infty]$
\begin{align*}
    \si(\Delta^{(n)}_{2})\subseteq    \si(\Delta^{(n)}_{p})\quad\mbox{and}\quad    \si(\Delta_{2})\subseteq    \si(\Delta_{p}).
\end{align*}

Let $G$ be a planar tessellation, see \cite{BP1,BP2} for background.
Denote by $F$ the set of faces and denote the degree of a face $f\in
F$, that is the number of vertices contained in face,  by $\deg(f)$.

We define the vertex curvature $\ka:X\to\R $ by
\begin{align*}
    \ka(x)=1-\frac{n(x)}{2}+\sum_{f\in F,x\in f}\frac{1}{\deg(f)}.
\end{align*}

The following theorem is a discrete analogue of
\cite[Proposition~1]{Sturm93}.

\begin{thm}If $\ka\ge0$, then it has  quadratic volume growth both with respect to $d_{n}$ and $m=n$ and with $d_{1}$ and $m\equiv 1$.
In particular, $\si(\Delta^{(n)}_{p})=\si(\Delta^{(n)}_{2})$ and
$\si(\Delta_{p})=\si(\Delta_{2})$ for all $p\in[1,\infty]$.
\end{thm}
\begin{proof}By \cite[Theorem~1.1]{HJL} the volume growth for $\ka\ge0$ is bounded
quadratically with respect to the measure $m=n$ and the metric
$d_{n}$. Moreover, it is direct to check that $\ka\geq0$ implies
that the vertex degree is bounded by $6$. Hence, we have bounded
geometry and thus $d_{n}$ and $d_{1}$ are equivalent and for
$m\equiv1$ we have $m\geq n/6$. Thus, the volume growth of the graph
is {quadratically bounded} also with respect to the measure
$m\equiv1$ and the metric $d_{1}$. The 'in particular' is now a
consequence of Theorem~\ref{main}.
\end{proof}

\begin{Remark} The result can be easily extended to
planar tessellation with finite total curvature, i.e., $\sum_{x\in
X}|\ka(x)|<\infty$ or equivalently vanishing curvature outside of a
finite set, see \cite{ChenChen} and also \cite{DeVosMohar}. This can
be seen as \cite[Theorem~1.1]{HJL} easily extends to the  finite
total curvature case.
\end{Remark}

The next theorem is a discrete analogue of
\cite[Proposition~2]{Sturm93}. We say that a graph with measure $m$
has\emph{ at least exponential volume growth} with respect to a
metric $d$ if for the corresponding distance balls $B_{r}(x)$ about
some vertex $x$,
\begin{align*}
    \mu=\liminf_{r\to\infty}\inf_{x\in X}\frac{1}{r}\log m(B_{r}(x))>0.
\end{align*}

\begin{thm}If  $\ka<0$, then the tessellation has at least exponential volume growth, both with respect to $d_{n}$ and $m=n$ and with $d_{1}$ and $m\equiv1$. Moreover, $\inf \si(\Delta^{(n)}_{1})\neq  \inf \si(\Delta^{(n)}_{2})$ and if we have additionally bounded geometry, then $ \inf\si(\Delta_{1})\neq  \inf\si(\Delta_{2})$
\end{thm}
\begin{proof} \cite[Theorem~C, Proposition~2.1]{Higuchi01} implies that if $\ka(x)<0$ for all $x\in X$, then $\sup_{x\in X}\ka(x)\leq -1/1806<0$. Hence, by \cite[Theorem~B]{Higuchi01} we have positive Cheeger constant, $\al>0$. Thus by a Cheeger inequality \cite{DK}
$\inf\si(\Delta_{2}^{(n)})\ge{{\al^2}/{2}}>0$ and by elementary
computations (cf. \cite{Keller10}) we infer
$\inf\si(\Delta_{2})\ge\inf\si(\Delta_{2}^{(n)})\cdot\inf_{x}\deg(x)>0$.
By \cite[Corollary~4.2]{HKW} this implies at least exponential
volume growth  {with respect to both metrics}. The second statement
follows along the lines of the proof of Proposition~\ref{p:alpha}.
\end{proof}

We end this section by an analogue theorem of
Theorem~\ref{t:superexponential}.

\begin{thm}If  $\inf_{K\subseteq X,\,\mathrm{\scriptsize{finite}}}\sup_{x\in X\setminus K}\ka(x)=-\infty$, then $\si(\Delta_{p}^{(n)})=\si(\Delta_{2}^{(n)})$ and $\si(\Delta_{p})=\si(\Delta_{2})$ for all $p\in(1,\infty)$. In particular, the spectrum is purely discrete and the eigenfunctions of $\Delta_{2}$ are contained in $\ell^{p}$ for all $p\in(1,\infty)$.
\end{thm}
\begin{proof}
By \cite[Theorem~3]{Keller10} the curvature assumption is is
equivalent to pure discrete spectrum of $\Delta_{2}$. Moreover, this
is equivalent to compact resolvent and compact semigroup on
$\ell^2$. Thus, the statement for $\Delta_{p}^{(n)}$ follows from
\cite[Theorem~1.6.3]{Davies89}. On the other hand, the curvature
assumption implies $\al_{\infty}=1$, \cite{Keller10}, and the
statement about $\Delta^{(n)}_{p}$ follows from
Theorem~\ref{t:superexponential}.
\end{proof}

{\bf Acknowledgements.} The first and second author thank J\"urgen
Jost for many inspiring discussions on this topic.The first author acknowledges the financial support of the Alexander von Humboldt Foundation.  The research leading to these results has received
funding from the European Research Council under the European
Union's Seventh Framework Programme (FP7/2007-2013) / ERC grant
agreement n$^\circ$ 267087. The second author
thanks Zhiqin Lu for introducing him this problem and stimulating
discussions at Fudan University. The third author thanks Daniel Lenz for generously sharing his knowledge on intrinsic metrics and  acknowledges the financial support of the German Science Foundation (DFG), the Golda Meir Fellowship and the Israel Science Foundation (grant no. 1105/10 and  no. 225/10).

\bibliography{LpIndependence}
\bibliographystyle{alpha}

\end{document}